\documentclass[12pt]{amsart}

\usepackage{amssymb,amsmath,amsxtra,amscd}
\usepackage{pstricks}
\usepackage{pst-tree}
\usepackage[mathscr]{eucal} 
\theoremstyle{plain}
\numberwithin{equation}{section}
\newtheorem{theorem}{Theorem}

\newtheorem{corollary}{Corollary}
\newtheorem{lemma}{Lemma}
\numberwithin{lemma}{section}
\numberwithin{theorem}{section}
\numberwithin{corollary}{section}
\numberwithin{proposition}{section}

\theoremstyle{definition}
\newtheorem{definition}{Definition}
\numberwithin{definition}{section}

\theoremstyle{remark}

\numberwithin{example}{section}

\newcommand{\nc}{\newcommand}
\nc{\mc}{\mathcal}
\nc{\OOn}{\OOperatorname}
\nc{\Cl}{Cl}
\nc{\drva}{\widehat{\OOmega}}
\nc{\spec}{\OOn{Spec}}
\nc{\AutO}{\OOn{Aut} \mc{O}}
\nc{\vac}{|0\rangle}
\nc{\Z}{\mathbb{Z}}
\nc{\zf}[1]{z^{\frac{1}{#1}}}
\nc{\wf}[1]{w^{\frac{1}{#1}}}
\nc{\Mt}{M^{\sigma}}

\nc{\gr}{\OOn{gr}}
\nc{\T}{\mathbb{T}}
\nc{\CT}{\mathbb{C}\{ \mathbb{T} \}}
\nc{\g}{\mathfrak{g}}
\nc{\n}{\mathfrak{n}}
\nc{\I}{\mathcal{I}}
\nc{\M}{M}
\nc{\OO}{\mathcal{O}}

\nc{\al}{\alpha}
\nc{\OOl}{\OOverline}

\begin{document}

\title{On the structure and representations of the insertion-elimination Lie algebra}

\author{Matthew Szczesny} \thanks{}
\address{Department of Mathematics  
         Boston University, Boston, MA 02215}
\email{szczesny@math.bu.edu}

\date{October 2007}

\begin{abstract}

We examine the structure of the insertion-elimination Lie algebra on rooted trees introduced in 
\cite{CK}. It possesses a triangular structure $\g = \n_+ \oplus \mathbb{C}.d \oplus \n_-$, like the Heisenberg, Virasoro, and affine algebras. We show in particular that it is simple, which in turn implies that it has no finite-dimensional representations. We consider a category of lowest-weight representations, and show that irreducible representations are uniquely determined by a "lowest weight" $\lambda \in \mathbb{C}$. We show that each irreducible representation is a quotient of a Verma-type object, which is generically irreducible.

\end{abstract}

\maketitle 

\section{Introduction}

The insertion-elimination Lie algebra $\g$ was introduced in \cite{CK} as a means of encoding the combinatorics of inserting and collapsing subgraphs of Feynman graphs, and the ways the two operations interact. A more abstract and universal description of these two operations is given in terms of rooted trees, which encode the hierarchy of subdivergences within a given Feynman graph, and it is this description that we adopt in this paper. More precisely, $\g$ is generated by two sets of operators $\{ D^+_t \}$, and $\{ D^-_t \}$, where $t$ runs over the set of all rooted trees, together with a grading operator $d$. In \cite{CK} $\g$ was defined in terms of its action on a natural representation $\CT$, where the latter denotes the vector space spanned by rooted trees. For $s \in \CT$, $D^+_t.s$ is a linear combination of the trees obtained by attaching $t$ to $s$ in all possible ways, whereas $D^-_t.s$ is a linear combination of all the trees obtained by pruning the tree $t$ from branches of $s$. $\n_+= \{ D^+_t \}$ and $\n_-=\{ D^-_t\}$ form two isomorphic nilpotent Lie subalgebras, and $\g$ has a triangular structure 
\[
\g = \n_+ \oplus \mathbb{C}.d \oplus \n_-
\]
as well as a natural $\mathbb{Z}$--grading by the number of vertices of the tree $t$. 
The Hopf algebra $U(\n_{\pm})$ is dual to Kreimer's Hopf algebra of rooted trees \cite{K}.

This note aims to estblish a few basic facts regarding the structure and representation theory of $\g$. We begin by showing that $\g$ is simple, which together with its infinite-dimensionality implies that it has no non-trivial finite-dimensional representations, and that any non-trivial representation is necessarily faithful. We then proceed to develop a highest-weight theory for $\g$ along the lines of \cite{K1, K2}. In particular, we show that every irreducible highest-weight representation of $\g$ is a quotient of a Verma-like module, and that these are generically irreducible. 

One can define a larger, "two-parameter" version of the insertion-elimination Lie algebra $\widetilde{\g}$, where operators are labelled by pairs of trees $D_{t_1, t_2}$ (roughly speaking, in acting on $\CT$, this operator replaces occurrences of $t_1$ by $t_2$). In the special case of ladder trees, $\widetilde{\g}$ was studied in \cite{M, KM1, KM2}. The finite-dimensional representations of the nilpotent subalgebras $\n_{\pm}$ as well as many other aspects of the Hopf algebra $U(\n_{\pm})$ were studied in \cite{F}. 

\bigskip

\noindent {\bf Acknowledgements:} The author would like to thank Dirk Kreimer for many illuminating conversations and explanations of renormalization as well as related topics. This work was supported by NSF grant DMS-0401619.

\section{The insertion-elimination Lie algebra on rooted trees} \label{basicfacts}

In this section, we review the construction of the insertion-elimination Lie algebra introduced in \cite{CK}, with some of the notational conventions introduced in \cite{M}.

Let $\T$ denote the set of rooted trees. An element $t \in \T$ is a tree (finite, one-dimensional contractible simplicial complex), with a distinguished vertex $r(t)$, called the root of $t$. Let $V(t)$ and $E(t)$ denote the set of vertices and edges of $t$,  and let 
$$
| t | = \#  V(t) 
$$
Let  $\CT$ denote the vector space spanned by rooted trees. It is naturally graded, 
\begin{equation} \label{CT}
\CT = \bigoplus_{n \in \mathbb{Z}_{\geq 0}} \CT_n
\end{equation}
where $\CT_n = \operatorname{span}\{ t \in \mathbb{T} \vert |t| = n \}$. $\CT_0$ is spanned by the empty tree, which we denote  by $\bf{1}$. We have \psset{levelsep=0.6cm}
\begin{center}
$\CT_0 = <1> $ \hskip 1cm $\CT_1=<\bullet> $ \hskip 1cm $\CT_2=$ $<$ \pstree{\Tr{$\bullet$}}{\Tr{$\bullet$}} $>$
\end{center}
\begin{center}
$\CT_3 = $ $<$ \pstree{\Tr{$\bullet$}}{\pstree{\Tr{$\bullet$}}{\Tr{$\bullet$}}} , \pstree{\Tr{$\bullet$}}{\Tr{$\bullet$}\Tr{$\bullet$}}$>$
\end{center}
where $<,>$ denotes span, and the root is the vertex at the top. 
If $e \in E(t)$, by a \emph{cut along} $e$ we mean the operation of cutting $e$ from $t$. This divides $t$ into two components - $R_c(t)$ containing the root, and $P_e(t)$, the remaining one. $R_e(t)$ and $P_e (t)$ are naturally rooted trees, with $r(R_c (t)) = r(t)$ and $r(P_e (t)) = $ (endpoint of e). Note that
$V(t) = V(R_e(t)) \cup V(P_e(t))$.

\bigskip

Let $\g$ denote the Lie algebra with generators $D^+_t, D^-_t, d$, $t \in \T$, and relations
\begin{equation}
\label{rel1}[D^+_{t_1}, D^{+}_{t_2}] = \sum_{v \in V(t_2)} D^{+}_{t_2 \cup_v t_1} - \sum_{v \in V(t_1)} D^{+}_{t_1 \cup_v t_2} 
\end{equation}
\begin{equation}
\label{rel2} [D^-_{t_1}, D^{-}_{t_2}] =   \sum_{v \in V(t_1)} D^{-}_{t_1 \cup_v t_2} 
-  \sum_{v \in V(t_2)} D^{-}_{t_2 \cup_v t_1} 
\end{equation}
\begin{equation}
\label{rel3} [D^-_{t_1}, D^{+}_{t_2}] =  \sum_{t \in \T} \alpha(t_1, t_2;t) D^{+}_t + \sum_{t \in T} \beta(t_1,t_2;t) D^{-}_t
\end{equation}
\begin{equation}
\label{rel4} [D^{-}_t, D^{+}_t] = d 
\end{equation}
\begin{equation}
\label{rel5} [d, D^{-}_t] = - |t| D^{-}_t 
\end{equation}
\begin{equation}
\label{rel6}  [d, D^{+}_t] = |t| D^{+}_t
\end{equation}

where for $s, t \in \T$, and $v \in V(s)$ $s \cup_v t$ denotes the rooted tree obtained by joining the root of $t$ to $s$ at the vertex $v$ via a single edge, and 

\begin{itemize}
\item $\alpha(t_t, t_2; t) = \# \{ e \in E(t_2) | R_e(t_2) = t, \, \, P_e(t_2) = t_1 \}$
\item $\beta(t_1, t_2; t) = \# \{ e \in E(t_1) | R_e(t_1) = t, \, \, P_e(t_1) = t_2 \}$
\end{itemize}

\bigskip

Thus, for example
\psset{levelsep=0.3cm, treesep=0.3cm}
\begin{align*}
 \psset{levelsep=0.3cm, treesep=0.3cm}
[D^+_{\bullet}, D^+_{ \pstree{\Tr{\bullet}}{\Tr{\bullet}\Tr{\bullet}} }] &= D^+_{\pstree{\Tr{\bullet}}{\Tr{\bullet} 
\Tr{\bullet} \Tr{\bullet }}} + 2 D^{+}_{ \pstree{\Tr{\bullet}}{\pstree{\Tr{\bullet}}{\Tr{\bullet}}\Tr{\bullet}}       }  - D^+_{\pstree{\Tr{\bullet}}{\pstree{\Tr{\bullet}}{ \Tr{\bullet} \Tr{\bullet}} }} \\
 \psset{levelsep=0.3cm, treesep=0.3cm}
[D^-_{\bullet}, D^-_{ \pstree{\Tr{\bullet}}{\Tr{\bullet}\Tr{\bullet}} }] &= -D^-_{\pstree{\Tr{\bullet}}{\Tr{\bullet} 
\Tr{\bullet} \Tr{\bullet }}} - 2 D^{-}_{ \pstree{\Tr{\bullet}}{\pstree{\Tr{\bullet}}{\Tr{\bullet}}\Tr{\bullet}}       }  + D^-_{\pstree{\Tr{\bullet}}{\pstree{\Tr{\bullet}}{ \Tr{\bullet} \Tr{\bullet}} }} \\
 \psset{levelsep=0.3cm, treesep=0.3cm}
[D^-_{\bullet}, D^+_{ \pstree{\Tr{\bullet}}{\Tr{\bullet}\Tr{\bullet}} }] &= 2 D^{+}_{ \pstree{\Tr{\bullet}}{\Tr{\bullet}}}
\end{align*}

$\g$ acts naturally on $\CT$ as follows. If $s \in \T$, viewed as an element of $\CT$, and $t \in \T$, then
\[
D^{+}_t (s) = \sum_{v \in V(s)} s \cup_v t
\]
\[
D^{-}_t (s) = \sum_{e \in E(s), P_e(s) = t} R_e (s)
\]
\[
d (s) = |s| s
\]

\section{Structure of $\g$}

Let $\n_+$ and $ \n_{-}$ be the Lie subalgebras s of $\g$ generated by $D^+_{t}$ and $D^-_t$ , $t \in \T$. We have a triangular decomposition
\begin{equation} \label{triangular}
\g = \n_+ \oplus \mathbb{C}. d \oplus \n_{-}
\end{equation}
The relations \ref{rel4}, \ref{rel5}, and \ref{rel6} imply that for every $t \in \T$
\[
\g^t =  <D^{+}_t, D^{-}_t, d>
\] 
forms a Lie subalgebra isomorphic to $\mathfrak{sl}_2$. We have that $\g_t \cap \g_s = \mathbb{C}. d$ if $s \ne t$.
Assigning degree $|t|$ to $D^{+}_t$, $-|t|$ to $D^{-}_t$, and $0$ to $d$ equips $\g$ with a $\mathbb{Z}$--grading. 
\[
\g = \bigoplus_{n \in \mathbb{Z}} \g_n
\]
$\g$ possesses an involution $\iota$, with
\[
\iota(D^+_t)=D^-_t \hspace{1in} \iota(D^-_t)=D^{+}_t \hspace{1in} \iota(d) = - d
\]
Thus $\iota$ is a gradation-reversing Lie algebra automorphism exchanging $\n_{+}$ and $\n_{-}$. 

\begin{theorem} \label{gissimple}
$\g$ is a simple Lie algebra
\end{theorem}

\begin{proof}
Suppose that $\I \subset \g$ is a proper Lie ideal. If $x \in \I$, let $x = \sum_i x_i, \; \;  x_i \in \g_i$ be its decomposition into homogenous components. We have
\[
[d,x] = \sum_{n}  n x_n
\]
which implies that $x_n \in I$ for every $n$ (because the Vandermonde determinant is invertible) i.e. $\I = \oplus_{n \in \mathbb{Z}} (\I \cap \g_n)$. Suppose now that $x_n \in \g_n, \; n > 0 $. We can write $x_n$ as a linear combination of $n$--vertex rooted trees
\begin{equation}
\label{xdecomp}
x_n = \sum_{t \in \T_n} \alpha_t \cdot t
\end{equation}
We proceed to show that $D^+_{\bullet} \in \I$, where $\bullet$ is the rooted tree with one vertex. Let $S(x_n) \subset \T_n$ be the subset of n-vertex trees occurring with a non-zero $\alpha_t$ in \ref{xdecomp}.
Given a rooted tree $t$, let $St(t)$ denote the set of rooted trees obtained by removing all the edges emanating from the root. Let 
$$
St(x_n) = \bigcup_{s \in S(x_n)} St(s)
$$
and let $\xi \in St(x_n)$ be of maximal degree. It is easy to see that $[D^-_\xi, x_n]$ is a non-zero element of $\g_{n-|\xi|}$. Starting with $x_n \in \n_+, \; x_n \neq 0$, and repeating this process if necessary, we eventually obtain a non-zero element of $\g_1 = < D^+_{\bullet} >$. Now, $[D^-_{\bullet}, D^+_{\bullet}] = d$, and since $[d, \g]=\g]$, this implies $\I=\g$.  We have thus shown that if $\I$ is proper, then 
$$\I \cap \n_+ = 0$$
Applying $\iota$ shows that $\I \cap \n_ - 0$ as well, and it is clear that $\I \cap \mathbb{C}.d = 0$. 

\end{proof}

We can now use this result to deduce a couple of facts about the representation theory of $\g$.

\begin{corollary}
If $V$ is a non-trivial representation of $\g$, then $V$ is faithful. 
\end{corollary}

\begin{corollary}
$\g$ has no non-trivial finite-dimensional representations.
\end{corollary}

The latter can also be easily deduced by analyzing the action of the $\mathfrak{sl}_2$ subalgebras $\g^t$ as follows. Suppose that $V$ is a finite-dimensional representation of $\g$. To show that $V$ is trivial, it suffices to show that it restricts to a trivial representation of $\g^t$ for every $t \in \T$. This in turn, will follow if we can show that for a \emph{single} tree $t \in \T$, $\g^t$ acts trivially, because this implies that $d$ acts trivially, and $\mathbb{C}.d \subset \g^t$ plays the role of the Cartan subalgebra.  Let 
\[
V = \bigoplus_{i=1\cdots k} V_{\delta_i}
\]
be a decomposition of $V$ into $d$--eigenspaces - i.e. if $v \in V_{\delta_i}$, then $d.v_i = \delta_i v$. Since $V$ is finite-dimensional, the set $\{\delta_i \}$ is bounded, and so lies in a disc of radius $R$ in $\mathbb{C}$. If $v \in V_{\delta_i}$ then $[d,D^{+}_t] = |t| D^{+}_t$ implies that $D^{+}_t . v \in V_{\delta_i+|t|}$. Choosing a $t \in \T$ such that $|t| > 2R$ shows that $D^{+}_t . v = 0$ for every $v \in V$.  
  
\subsection{Lowest-weight representations of $\g$}

We begin by examining the "defining" representation $\CT$ of $\g$ introduced in section \ref{basicfacts}. Its decomposition into $d$--eigenspaces is given by \ref{CT}. Given a representation $V$ of $\g$ on which $d$ is diagonalizable, with finite-dimensional eigenspaces, and writing 
\[
V = \bigoplus_{\delta} V_{\delta}
\] 
for this decomposition, we define the emph{character} of $V$, $char(V,q)$ to be the formal series
\[
char(V,q) = \sum_\delta dim(V_\delta) q^\delta
\]
The case $V=\CT$, where $dim(V_n)$ is the number of rooted trees on $n$ vertices, suggests that representations of $\g$ may contain interesting combinatorial information.  The triangular structure  \ref{triangular}  of $\g$ suggests that a theory of highest-- or lowest--weight representations may be appropriate.

\begin{definition} \label{lw}
We say that a representation $V$ of $\g$ is \emph{lowest--weight} if the following properties hold
\begin{enumerate}
\item $V = \oplus V_{\delta}$ is a direct sum of finite-dimensional eigenspaces for $d$.
\item The eigenvalues $\delta$ are bounded in the sense that there exists $L \in \mathbb{R}$ such that $Re(\delta) \geq L$.
\end{enumerate}
\end{definition} 

We call the $\delta$ the \emph{weights} of the representation, and category of such representations $\OO$. If $V \in \OO$,  we say $v \in V_{\delta}$ is a \emph{lowest-weight vector} if $\n_- v = 0$.  Since $D^{-}_t$ decreases the weight of a vector by $|t|$, and the weights all lie in a half-plane, it is clear that every $V \in \OO$ contains a lowest-weight vector. 

\medskip

\noindent Recall that a representation $V$ of $\g$ is \emph{indecomposable} if it cannot be written as $V = V_1 \oplus V_2$ for two non-zero representations. Let $U(\mathfrak{h})$ denote the universal enveloping algebra of a Lie algebra $\mathfrak{h}$.

\begin{lemma}
   If $v \in V_\lambda$ is a lowest-weight vector, then $U(\n_+).v$ is an indecomposable representation of $\g$
\end{lemma}

\begin{proof}
   $U(\g).v$ is clearly the smallest sub-representation of $V$ containing $v$. The decomposition \ref{triangular} together with the PBW theorem implies that
   \[
   U(\g) = U(\n_+)\otimes \mathbb{C}[d] \otimes U(\n_-)
   \]
   Because $v$ is a lowest-weight vector, $\mathbb{C}[d] \otimes U(\n_-).v = \mathbb{C}.v$. It follows that $U(\g).v = U(\n_+).v$. That the latter is indecomposable follows from the fact that in  $U(\n_+).v$, the weight space corresponding to $\lambda$ is one-dimensional, and so if $U(\n_+).v=V_1\oplus V_2$, then $v \in V_1$ or $v \in V_2$. 
\end{proof}

Observe that 
$$U(\n_+).v = \oplus (U(\n_+).v)_{\lambda+k}, \; \; k \in \mathbb{Z}_{\geq 0}$$
where $(U(\n_+).v)_{\lambda+k}$ is spanned by monomials of the form
\begin{equation} \label{PBW}
D^{+}_{t_1} D^{+}_{t_2} \cdots D^{+}_{t_i} . v
\end{equation}
with $|t_1|+\cdots |t_i| = k$. 

The category $\mc{O}$ contains Verma-like modules. For $\lambda \in \mathbb{C}$, let $\mathbb{C}_{\lambda}$ denote the one-dimensional representation of $\mathbb{C}.d\oplus \n_{-}$ on which $\n_-$ acts trivially, and $d$ acts by multiplication by $\lambda$.

\begin{definition}
The $\g$--module
\[
W(\lambda) = U(\g) \underset{\mathbb{C}[d] \otimes U(\n_-)}{\otimes} \mathbb{C_\lambda}
\]
will be called \emph{the Verma module} of lowest weight $\lambda$.
\end{definition}

\noindent Choosing an ordering on trees yields a PBW basis for $\n_+$, and thus also a basis of the form \ref{PBW} for $W(\lambda)$.

Given a representation $V \in \mc{O}$, and a lowest weight vector $v \in V_\lambda$, we obtain a map of representations
\begin{equation} \label{univprop}
W(\lambda) \mapsto V
\end{equation}
\[
\mathbf{1} \mapsto v
\]

\begin{lemma}
If $V \in \mc{O}$ is an irreducible representation, then $V$ is the quotient of a Verma module.
\end{lemma}

\begin{proof}
Since $V \in \mc{O}$, $V$ possesses a lowest-weight vector $v \in V_\lambda$ for some $\lambda \in \mathbb{C}$. Since $V$ is irreducible, $V = U(\g).v=U(\n_+).v$. The latter is a quotient of $W(\lambda)$.  
\end{proof}

\bigskip

\noindent We have
\begin{align*}
Char(W(\lambda)) &= q^{\lambda} \sum_{n \in \mathbb{Z}_{\geq 0}} dim(\CT_n) q^{n} \\
                                  &= q^{\lambda} \prod_{n \in \mathbb{Z}_{\geq 0}} \frac{1}{{(1 - q^n)}^{P(n)}}\\
\end{align*}
where $P(n)$ is the number of primitive elements of degree $n$ in $\mathcal{H}_{K}$.

\bigskip

\subsection{Irreducibility of $W(\lambda)$ }

It is a natural question whether $W(\lambda)$ is irreducible. In this section we prove the following result:

\begin{theorem}
For $\lambda$ outside a countable subset of $\mathbb{C}$ containing $0$, $W(\lambda)$ is irreducible. 
\end{theorem}
\begin{proof}
Let $v \neq 0$ be a basis for $W(\lambda)_\lambda$. 
$W(\lambda)$ contains a proper sub-representation if and only if contains a lowest-weight vector $w$ such that $w \notin \mathbb{C}.v$. In $W(0)$, $D^+_{\bullet} .v \in W(0)_1$ is a lowest-weight vector, since
\[
D^-_{\bullet} D^{+}_{\bullet}.v = D^{+}_{\bullet} D^-_{\bullet}.v + d.v = 0
\]
and $D^{-}_t.v =0$ for all $t \in \T$ with $|t| \geq 2$ by degree considerations. It follows that $W(0)$ is not irreducible. 

If $I=(t_1, \cdots, t_k)$ is a $k$--tuple of trees such that 
$$t_1 \preceq t_2 \preceq \cdots \preceq t_k $$ in the chosen order, let 
$D^+_I.v$ denote the vector
\begin{equation} \label{vv}
D^+_{t_k} \cdots D^+_{t_1}.v \; \in W(\lambda)
\end{equation}
$w \in W(\lambda)_{\lambda+n}$ is a lowest-weight vector if and only if 
\begin{equation} \label{lw}
D^-_t.w=0
\end{equation}
for all $t$ such that $|t| \leq n$. Writing $w$ in the basis \ref{vv}
\[
w = \sum_{|I|=n} \alpha_I D^+_{I}.v
\]
the conditions \ref{lw} translate into a system of equations for the coefficients $\alpha_I$. For example, if $w \in W(\lambda)_{\lambda+2}$, then 
\[ \psset{levelsep=0.3cm, treesep=0.3cm}
w=\alpha_1 D^+_{ \pstree{\Tr{\bullet}}{\Tr{\bullet}}}.v + \alpha_2 D^{+}_{\bullet} D^{+}_{\bullet}.v
\]
and conditions  $D^-_{\pstree{\Tr{\bullet}}{\Tr{\bullet}}}.v =0 $, $D^-_{\bullet}.w=0$ translate into
\begin{align*}
\lambda \alpha_1 + \lambda \alpha_2 &= 0 \\
\alpha_1 + (2\lambda + 1) \alpha_2 &=0
\end{align*}
The determinant of the corresponding matrix is $2 \lambda^2$, and so for $\lambda \neq 0$, there is no lowest-weight vector $w \in W(\lambda)_{\lambda+2}$. For a general $n$, the system can be written in the form
\[
(A + \lambda B) [\alpha_I] = 0
\]
where $A$ and $B$ are matrices whose entries are non-negative integers. Let
\[
f_n (\lambda) = dim(Ker(A+\lambda B))
\]
Then for every $r \in \mathbb{N}$
\[
S_{n,r} = \{ \lambda \in \mathbb{C} | f_n(\lambda) \geq r \}
\]
if proper, is a finite subset of $\mathbb{C}$, since the condition is equivalent to the vanishing a finite collection of sub-determinants, each of which is a polynomial in $\lambda$. The set of $\lambda \in \mathbb{C}$ for which $W(\lambda)$ is irreducible is therefore
\[
\bigcup_{n \in \mathbb{N}} \{ \mathbb{C} \backslash S_{n,1} \}
\]
The theorem will follow if $S_{n,1}$ is proper for each $n \in \mathbb{N}$.
This follows from the following Lemma.

\end{proof}

\begin{lemma} \label{Z1}
$Z(1)$ is irreducible.
\end{lemma}

\begin{proof}
We begin by examining the representation $\CT$. The degree $0$ subspace $\mathbb{C}.1$ is a trivial representation of $\g$. Let $\M$ denote the quotient $\CT/\mathbb{C}.1$. It is easily seen that the exact sequence
\[
0 \mapsto \mathbb{C} \mapsto \CT \mapsto \M \mapsto 0
\]
is non-split. $\M$ has highest weight $1$, and the subspace $\M_1$ can be identified with the span of the tree on one vertex $\bullet$. By the universal property of Verma modules, \ref{univprop} we have a map 
\begin{equation} \label{surject}
W(1) \mapsto M
\end{equation}
sending the lowest-weight vector of $W(1)$ to $\bullet$. Now, $W(1)_n$ is spanned by all vectors \ref{vv} such that $|t_1|+ \cdots |t_k| = n-1$, and so can be identified with the set of forests on $n-1$ vertices, while $M_n$ can be identified with $\CT_n$. The operation of adding a root to a forest on $n-1$ vertices to produce a rooted tree with $n$ vertices yields an isomorphism $W(1)_{n} \cong M_n$. Thus, if the map \ref{surject} is a surjection, it is an isomorphism. This in turn, follows from the fact that $M$ is irreducible. 

It suffices to show that $M_n$ contains no lowest-weight vectors for $n > 1$.  This follows from an argument similar to the one used to prove \ref{gissimple}. Let $w \in M_n$, and write
\[
w = \alpha_1 t_1 + \cdots \alpha_k t_k
\]
where $|t_i|=n$ and we may assume that $\alpha_i \neq 0$. In the notation of \ref{gissimple}, let $\xi \in St(w)$ be of maximal degree. Then
\[
D^-_{\xi}. w \neq 0
\]
Thus, $\M$ is irreducible, and hence isomorphic to $W(1)$ by the map \ref{surject}.
\end{proof}

\end{document}